\documentclass[11pt,a4paper]{amsart}
\usepackage[english]{babel}
\usepackage{amssymb,amsmath,amsthm}
\usepackage{amsfonts}
\usepackage{graphicx}
\usepackage{float}

\usepackage{verbatim}
\usepackage{hyperref}

\usepackage[normalem]{ulem}
\usepackage[font=small,labelfont=bf]{caption}
\newtheorem{theorem}{Theorem}
\newtheorem{lemma}[theorem]{Lemma}

\newtheorem{proposition}[theorem]{Proposition} 
\theoremstyle{definition} 
 
\newtheorem{remark}[theorem]{Remark}
\newtheorem{example}[theorem]{Example}

\newtheorem*{theorem*}{Theorem}
\newtheorem{theoremA}{Theorem}
\newtheorem*{theoremA*}{Theorem A}

\usepackage{xcolor} 
\newcommand{\ro}[1]{ \textcolor{blue}{#1} \normalsize}

\title[Limit cycles of linear piecewise ODEs in the cylinder]{Finitude of limit cycles of linear piecewise ODEs in the cylinder}

\author{J.L. Bravo}
\address{Departamento de Matematicas, Universidad de Extremadura, 06006 Badajoz, Spain}
\email{trinidad@unex.es}
\author{R. Trinidad-Forte}
\address{Departamento de Ciencias Matemáticas e Informática, Universitat de les
Illes Balears, 07122 Palma, Spain}
\email{roberto.trinidad@uib.cat}

\thanks{The authors are partially supported by the project PID2023-151974NB-I00 funded by MICIU/AEI/10.13039/501100011033/FEDER, UE, and by ‘ERDF A way of making Europe’. This work has been partially funded by the Junta de Extremadura and by ‘ERDF A way of making Europe’ through
project GR24042 (J.L.B.).}

\begin{document}

\begin{abstract}
Let $x'=S(t,x)$ be a differential equation in the cylinder, linear piecewise in $x$ and with trigonometric coefficients in $t$. In this paper, we provide an upper bound on the number of limit cycles in terms of the number of regions of the piecewise equation and the degree of the coefficients, that is, an analogue of Hilbert's 16th problem in this context. 
\end{abstract}

\maketitle

\section{Motivations and main result}

The second part of Hilbert's sixteenth problem asks to obtain an upper bound $\mathcal{H}(n)$ of the number of limit cycles of polynomial planar vector fields in terms of their degree $n$. The question is extremely challenging and remains unsolved to this day. Substantial effort has been devoted to particular families of planar differential equations, such as quadratic systems, cubic systems, Liénard-type equations, Abel equations, among others. We refer to \cite{Gasull2,Ilyashenko,Jibin} and the references therein for a deeper exposition on the study of the limit cycles of planar polynomial ODEs.

A particular version of Hilbert's sixteenth problem can be stated in one-dimensional non-autonomous periodic differential equations, that is, families of equations of the form 
\begin{equation}\label{eq:0}
x' = S(t,x),\quad (t,x)\in\mathbb{R}^2,
\end{equation}
where $S$ is $2\pi$-periodic in $t$. 
In this context, assuming existence and uniqueness of solutions, a solution $x$ is periodic if and only if $x(0)=x(2\pi)$ (see, e.g.,~\cite{Forte}). A {\em limit cycle} is defined as a periodic solution that is isolated within the set of all periodic solutions of \eqref{eq:0}. 
In this context, Hilbert's sixteenth problem is to bound the number of limit cycles of ~\eqref{eq:0} in terms of certain degrees of the family that will be defined later.
These systems have been profusely studied due to their relation with related problems for higher-dimensional piecewise systems~\cite{JL3zonas,Buzzi,  
Carmona}, planar autonomous polynomial differential equations~\cite{Conti,Devlin, Prohens, Lloyd}, and also for its application to problems in control theory~\cite{Carmona, Fossas} or to other real-world models~\cite{Benardete, Harko}.

The most studied version is Hilbert's sixteenth problem when $S$ is polynomial in $x$, known as Smale-Pugh's problem~\cite{Smale}. It is known that the bound can not depend only on the degree as a polynomial in $x$, as Lins-Neto in \cite{Neto} proved that the number of limit cycles for Abel equations with trigonometric polynomial coefficients is not bounded. Therefore, to obtain a uniform upper bound, it is necessary to consider also the degree of the coefficients. Solving the  Hilbert's sixteenth problem in this version would solve classical Hilbert sixteenth problem for quadratic vector fields, among others (see e.g. \cite{Coppel1966,Devlin}).

The Hilbert number of~\eqref{eq:0}, $\mathcal{H}(n,m)$, is the maximum number of limit cycles that~\eqref{eq:0} has for fixed $n,m$, where $n$ is the degree of $S$ as a polynomial in $x$ and $m$ is the maximum of the degree of the coefficients. Hilbert's sixteenth problem can be defined as computing $\mathcal{H}(n,m)$ or at least providing an upper bound.

Hilbert numbers has only been obtained for certain families (see e.g.\cite{Andersen, GasullGuillamon, GasullLlibre,HuangTianZhao, Neto, Panov,YuHuangLiu}). Moreover, it is not even known if the number of limit cycles for a particular equation is finite, except when the leading coefficient is constant~\cite{Ilyashenko2000}. We refer the reader to \cite{Gasull} more details and references. 

Another case that is recently gathering considerable attention is when $S$ is piecewise-linear in $x$, and the coefficients of the linear functions are trigonometric monomials in $t$ (see e.g. \cite{Ojeda,Forte,Coll,Gasull,HuangJin2019,Huang,Tian}). These equations have interest by themselves, but also appear 
studying certain families of linear-piecewise systems in higher dimensions (see e.g. \cite{JL3zonas,Carmona,Huang}).

To obtain an analogue to Hilbert's sixteenth problem in this setting, we need to give a notion of degree with respect to $x$. As $S$ is piecewise-linear in $x$, a good candidate is the number of linear pieces. Again, there is no upper bound on the number of limit cycles just in terms of the number of linear pieces; in \cite{Coll} it is proved that, given any natural number $k \geq 2$, for $\varepsilon$ sufficiently small, the differential equation
\[
x' = 2\pi \sin (2\pi t) + \varepsilon \cos (2k \pi t)|x|
\]
has at least $k - 2$ limit cycles. Therefore, the upper bound should be obtained in terms of the number of regions and the degree of the trigonometric coefficients.

We define the Hilbert numbers of~\eqref{eq:0}, denoted by $\mathcal{H}(n,m)$, as the maximum number of limit cycles that equation~\eqref{eq:0}, with $n+1$ regions and trigonometric coefficients of degree $m$, can exhibit. In this context, Hilbert’s sixteenth problem consists in finding an upper bound for the number of limit cycles of~\eqref{eq:0} in terms of the number of regions and the degree of the trigonometric coefficients.

So far, Hilbert numbers have been determined only for certain specific families. For instance, if~\eqref{eq:0} is of the form
\[
x' = a(t)|x| + b(t),
\]
and either $a$ or $b$ has constant sign, then $\mathcal{H}=2$~\cite{Gasull}. Moreover, when both $a$ and $b$ are trigonometric polynomials of degree one, it has been shown that $\mathcal{H}\leq 23$~\cite{Ojeda}.

In~\cite{Carmona}, the authors prove that continuous, observable, and non-controllable three-dimensional symmetric piecewise linear systems with two or three regions can be reduced to equation~\eqref{eq:0} with
\[
S(t,x)=F(x)+\mu\sin(t),
\]
where $F$ is a piecewise-linear function with two or three pieces, respectively. They show that in the two-region case, the maximum number of limit cycles is two. For the three-region case, the system is known to admit at least five limit cycles, and the total number is proved to be finite~\cite{JL3zonas}. In particular, $\mathcal{H}(2,1)\geq 5$.

Finally, by~\cite{Coll}, it is known that $\mathcal{H}(1,m)\geq m-2$. This bound was improved in~\cite{HuangJin2019} to $\mathcal{H}(1,m)\geq m+2$ when $S$ is continuous, and to $\mathcal{H}(1,m)\geq m+4$ in the discontinuous case.

This problem can also be extended to the case where $S$ is piecewise polynomial in $x$, for which some partial results are available~\cite{Huang,Tian}. However, the general setting would include the Abel equation as a particular case, and for this reason, it exceeds the aims of the present work.

The objective of this work is to obtain an upper bound on the Hilbert numbers associated with~\eqref{eq:0} when $S$ is piecewise linear in $x$ and has trigonometric polynomial coefficients in $t$. More precisely, equation~\eqref{eq:0} is assumed to be of the form
\begin{equation}\label{eq:mainPCW}
x' = S(x,t) =
\begin{cases}
a_{n+1}(t)\,x + b_{n+1}(t), & x \ge x_n,\\[1mm]
a_i(t)\,x + b_i(t), & x_{i-1} \le x \le x_i,\quad i=2,\dots,n,\\[1mm]
a_1(t)\,x + b_1(t), & x \le x_1,
\end{cases}
\end{equation}
where $a_1,\dots,a_{n+1}$ and $b_1,\dots,b_{n+1}$ belong to $\mathbb{R}[\sin(t),\cos(t)]$. In order to consider the most general setting, we do not assume that $S$ is continuous.

Throughout this work we restrict our attention to \emph{crossing solutions}, that is, absolutely continuous functions $u(t)$ satisfying
\begin{enumerate}
\item if $u(t)\notin \{x_1,\ldots,x_n\}$, then $u'(t) = S(t,u(t))$;
\item if $u(t)=x_k$ for some $k\in\{1,\ldots,n\}$, then
\[
\bigl(a_k(t)x_k + b_k(t)\bigr)\bigl(a_{k+1}(t)x_k + b_{k+1}(t)\bigr) > 0.
\]
\end{enumerate}
A \emph{crossing limit cycle} is defined as a periodic crossing solution that is isolated within the set of periodic crossing solutions. We will show that, when $S$ is continuous, any upper bound on the number of crossing limit cycles also provides an upper bound on the number of limit cycles of~\eqref{eq:0}.

The same techniques could be applied when considering solutions in the sense of Filippov~\cite{Filippov}. However, due to the possible lack of uniqueness, a precise notion of limit cycle would be required, and this situation is not addressed here.

Let \[M = \max \{\deg(a_i), \deg(b_i)\}_{i=1}^{n+1}.\] 
The main result of this paper provides an explicit upper bound for the number of crossing limit cycles of equation~\eqref{eq:mainPCW} in terms of $n,M$.

\begin{theoremA}\label{thm:A}

For fixed $n,M$, the maximum number of crossing limit cycles of \eqref{eq:mainPCW} satisfies 
    \[ \mathcal{H}(n,M) \leq
2Mn \left( 2^{4Mn(8Mn-1)}\,(6M^2n + 2Mn+1 )^{8Mn}\right) + 2.
\] 
\end{theoremA}

The paper is organised as follows. In Section~2, we introduce Khovanskiĭ theory, which constitutes the main tool used to obtain the upper bounds. In Section~3, we analyse a simpler version of equation~\eqref{eq:mainPCW} with only two regions. This analysis allows us to highlight the main difficulties arising in the study of such differential equations, while at the same time providing a sharper bound for $\mathcal{H}(1,m)$. Finally, in Section~4, we prove the main result by extending the techniques developed in Section~3.

\section{Pfaffian functions and Khovanskiĭ theory}

This section gathers the necessary tools for studying the number of solutions of systems of equations involving transcendental functions. We introduce the theoretical framework required to apply Khovanskiĭ’s theorem, which provides bounds on the number of non-degenerate solutions of systems defined by Pfaffian functions.

Khovanskiĭ’s original theorem on fewnomials \cite{Khovanskiĭ} provides a bound on the number of isolated real solutions of systems of polynomial equations in terms of the number of distinct monomials, independently of their degrees. Nevertheless, when one considers systems like \eqref{eq:mainPCW}, where $a_1,\dots,a_{n+1},b_1,\dots,b_{n+1}$ are general trigonometric polynomials, the resulting functions derived from the formula of the solution of linear ODEs involve compositions such as exponentials of trigonometric terms. These go beyond the scope of the fewnomial theorem, which provides an upper bound for the number of solutions of a system of equations with trigonometric and exponential functions, but not a composition of such functions.

However, Khovanskiĭ extended his theory to Pfaffian functions: analytic functions whose derivatives satisfy a triangular system of differential equations with polynomial coefficients. These functions were first introduced by Khovanskiĭ (see \cite{Khovanskiĭ,Khovanskiĭ2}) and have later been studied by other authors (see \cite{Gabrielov,gabrielov1995}). Khovanskiĭ realized that many analytic transcendental functions satisfy systems of equations that have a controlled algebraic structure. 

By correctly defining the degree of the polynomials involved and the dimension of the algebraic structure, he provided an upper bound on the number of non-degenerate real solutions for a system of a broader type of functions (see Theorem \ref{Khov}).

The section aims to prove that fixed  $a,b\in\mathbb{R}[\sin\,t,\cos\,t]$ and $x_1,x_2\in\mathbb{R}$ the functions defined by $(t_1,t_2)\in (-\pi,\pi)^2 \to T(t_1,t_2)$, where
\begin{equation}\label{eq:transition}
T(t_1,t_2):=x_{1}e^{-\int_{0}^{t_1}a(s)ds}
+\int_{t_1}^{t_2}b(t)e^{-\int_{0}^{t}a(\tau)d\tau}dt
-x_{2}e^{-\int_{0}^{t_2}a(s)ds},
\end{equation} which are related to the times of the solution of the linear non-autonomous ODE $x'=a x + b$ passing through two given values, are Pfaffian, so Kovanskiĭ's results apply. 
\medskip

The following definitions and results are extracted from~\cite{Gabrielov}. 
\medskip

A \textit{Pfaffian chain} of order $r \geq 0$ and degree $\alpha \geq 1$ in an open domain $G \subset \mathbb{R}^n$ is a sequence of analytic functions $f_1, \ldots, f_r$ in $G$ satisfying a chain of differential equations
\begin{equation*}
df_j(x) = \sum_{i=1}^{n} g_{ij}(x, f_1(x), \ldots, f_j(x)) \, dx_i \quad \text{for } 1 \leq j \leq r,
\end{equation*}
where $g_{ij}(x, y_1, \ldots, y_j)$ are polynomials in $x = (x_1, \ldots, x_n), y_1, \ldots, y_j$ of degrees not exceeding $\alpha$.

A function $f(x) = P(x, f_1(x), \ldots, f_r(x))$, where $P(x, y_1, \ldots, y_r)$ is a polynomial of degree not exceeding $\beta \geq 1$, is called a \textit{Pfaffian function} of order $r$ and degree $(\alpha, \beta)$.    
\medskip

Examples of Pfaffian functions are polynomials, rational functions, exponentials or trigonometric functions (we refer again to~\cite{Gabrielov} for more details).

\begin{enumerate}
\item Polynomials of degree $n$ are Pfaffian functions of order $0$ and degree $(1,n)$.

\item The function $f(x) = 1/x$ is a Pfaffian function of order $1$ and degree $(2,1)$ due to the chain  $f'(x) = -f^2(x)$.

\item The exponential function $f(x) = \exp(x)$ is a Pfaffian function of order $1$ and degree $(1,1)$, since $f'(x) = f(x)$.

\item The functions $f_1(x) = \sin(x)$ and $f_2(x) = \cos(x)$ are Pfaffian functions of order $2$ and degrees $(2,2)$ and $(2,3)$ respectively for $x\in(-\pi,\pi)$: If we take $f(x) = \tan(x/2)$ and $g(x) = 1 / (1 + f^2(x))$, then 

\begin{equation*}
\begin{split}
&f'(x) = \frac{1+f^2(x)}{2},\quad g'(x) = -\dfrac{2f(x)f'(x)}{(1+f^2(x))^2} = -f(x)g(x), \\ 
&f_1'(x) = \dfrac{1 - f^2(x)}{1+f^2(x)} = \left(1-f^2(x)\right)g(x),\quad  f_2'(x) = -\dfrac{2f(x)}{1+f^2(x)} =  -2f(x)g(x).
\end{split}
\end{equation*}
\end{enumerate}

The next result proves that Pfaffian functions possess a ring structure and shows how the order and degree are affected by the sum and product of Pfaffian functions.

\begin{lemma}[\cite{Gabrielov}]
The sum (resp. product) of two Pfaffian functions $f_1$ and $f_2$ of orders $r_1$ and $r_2$ and degrees $(\alpha_1, \beta_1)$ and $(\alpha_2, \beta_2)$ respectively, is a Pfaffian function of order $r_1 + r_2$ and degree $(\alpha, \max\{\beta_1, \beta_2\})$ (resp. $(\alpha, \beta_1 + \beta_2)$), where $\alpha = \max\{\alpha_1, \alpha_2\}$. If the two functions are defined by the same Pfaffian chain of order $r$, then the orders of the sum and of the product are both equal to $r$.    
\end{lemma}

The following theorem will be the key tool for bounding the number of crossing limit cycles, as it provides an upper bound on the number of real non-degenerate solutions of a system of equations in terms of the degree of the Pfaffian functions and the length of the Pfaffian chain.

\begin{theorem}[\cite{Khovanskiĭ}]\label{Khov}
    Consider a system of equations $f_1=\ldots=f_n = 0,$ where $f_i,\, 1\leq i \leq n$ are Pfaffian functions in a domain $G \subset \mathbb{R}^n$, having a common Pfaffian chain of order $r$ and degrees $(\alpha, \beta_i)$ respectively. Then the number of non-degenerate solutions of this system does not exceed 
    \begin{equation*}
        \mathcal{M}(n,r,\alpha,\beta_1,\ldots,\beta_n) := 2^{r(r-1)/2}\beta_1\dots \beta_n (\min \{n,r\}\alpha + \beta_1 + \dots + \beta_n - n + 1)^r.
\end{equation*}   
\end{theorem}

Since our goal is to give an upper bound on the number of limit cycles of equation \eqref{eq:mainPCW}, the second key ingredient will be to prove that the solutions of a linear ODE with trigonometric coefficients are Pfaffian functions. More concretely, in the following result, we prove that the ``time of flight'' of the solution of a linear ODE is Pfaffian.

\begin{lemma}\label{lemma:0}
Fixed  $a,b\in\mathbb{R}[\sin\,t,\cos\,t]$ of degree at most $m$, and $x_1,x_2\in\mathbb{R}$ the functions $t\to T(t,t_2)$ and $t\to T(t_1,t)$, where $T$ is defined by~\eqref{eq:transition}, are Pfaffian of order 4 and degree $(1+3m,1)$, for every $t_1,t_2\in (-\pi,\pi)$.
\end{lemma}
\begin{proof}
We will prove it only for $t\to T(t_1,t)$, as the other function is analogous. First, let us define a Pfaffian chain for our function:
\[
\begin{split}
f_1(t) = \tan\left(\frac{t}{2}\right),
\quad &f_2(t) = \frac{1}{1+f_1^2(t)},\\
f_3(t) = \exp\left(-A(t)\right), \quad 
&f_4(t) = \int_{t_1}^t b(s) f_3(s)\,ds,
\end{split}
\]
where $A(t) := \int_0^t a(t)\,dt$.

The first two functions are trivially part of a Pfaffian chain, as 
\[f_1'(t) = \frac{1 + f_1^2(t)}{2},\quad 
f_2'(t) = - \frac{2 f_1(t) f_1'(t)}{(1+f_1^2(t))^{2}} = -f_1(t) f_2(t).\]
To prove that $f_3$ extends the chain, let us show that
\[
\sin(kt) = \frac{P_k(\tan(t/2))}{(1+\tan^2(t/2))^k}, \quad 
\cos(kt) = \frac{Q_k(\tan(t/2))}{(1+\tan^2(t/2))^k},
\]
where $P_k,Q_k$ are polynomials of degree at most $2k$ with respect to the chain. 

Let us realize that 
\[
e^{ikt} = \cos(kt) + i\sin(kt) =\left( \cos(t) + i\sin(t) \right)^k = \left( \dfrac{1-\tan^2(\frac{t}{2}) + 2i\tan(\frac{t}{2})}{1+\tan^2(\frac{t}{2})} \right)^k,
\]
and so
\[
    \cos(kt) = \Re (e^{ikt}) = \dfrac{\Re (1-\tan^2(\frac{t}{2}) + 2i\tan(\frac{t}{2}))^k}{\left( 1+\tan^2(t/2) \right)^k} = \dfrac{Q_k(\tan(\frac{t}{2}))}{\left( 1+\tan^2(t/2) \right)^k},
\]
where $Q_k(\tan(\frac{t}{2}))$ is the real part of the polynomial $(1-\tan^2(\frac{t}{2}) + 2i\tan(\frac{t}{2}))^k$. We can compute $Q_k$ by Newton's binomial theorem. Note that $Q_k$ has degree $2k$. Analogously with the sine, so we may write

\[
\sin(kt) = \dfrac{P_k(\tan(\frac{t}{2}))}{1+\tan^2(t/2)} = P_k(f_1) f_2^k, \quad \cos(kt) = \dfrac{Q_k(\tan(\frac{t}{2}))}{1+\tan^2(t/2)} = Q_k(f_1)f_2^k.   
\]

Therefore, if $a(t) = a_0 + \sum_{k=1}^m a_k \cos(k\,t) + \sum_{k=1}^m \tilde{a}_k \sin(k\,t)$, then
\[
\begin{split}
f_3'(t) &= -f_3(t) a(t) 
\\&= -f_3(t)\left(
a_0 + \sum_{k=1}^m a_k Q_k(f_1(t))f_2^k(t) + \sum_{k=1}^m \tilde{a}_k P_k(f_1(t))f_2^k(t)\right),
\end{split}
\]
thus a polynomial in $f_1,f_2,f_3$ of degree at most $3m+1$.
\medskip 

Finally, 
\[
f_4'(t) = b(t) f_3(t), 
\]
and arguing as above, $b(t)$ is a polynomial in $f_1,f_2$ of degree at most $3m$, so $bf_3$ has degree at most $3m +1$ as a polynomial in $f_1,f_2,f_3$. 
\medskip 

To conclude, note that
\[
T(t_1,t) = x_1 f_3(t_1) + f_4(t) - x_2 f_3(t).
\]

\end{proof}

\section{Linear piecewise ODE with two regions}

In this section, we obtain an upper bound on the number of limit cycles of a piecewise linear ODE with two regions. We use the structure to find a bound better than the one in the general case, which is included in the next section.

Consider the piecewise-linear differential equation with two regions. By a translation in $x$, it can be written as
\begin{equation}\label{eq:main2}
x' = S(t,x) =
\begin{cases}
a^+(t)x + b^+(t), \quad x \geq 0, \\
a^-(t)x + b^-(t), \quad x < 0,
\end{cases}
\end{equation}  
where $a^+,\,a^-,\,b^+,\,b^-$ are trigonometric polynomials with real coefficients. Let $m=\max \{ \deg (b^+), \deg (b^-) \}$. Recall that continuity is not assumed and that only crossing limit cycles are considered.

If $b^+(t)$ or $b^-(t)$ have constant sign, then crossing limit cycles can not cross the line $x=0$, and the problem can be divided into two linear ones. Consequently, in the following, we assume that both trigonometric polynomials change sign.

The goal is to reduce the study of the number of crossing limit cycles to the number of solutions of a certain system of equations, and then apply Theorem \ref{Khov}. Firstly, we analyze the qualitative structure of the solutions in a general setting, and then we derive necessary and sufficient conditions for having a limit cycle.

\subsection{Structure of the set ofcrossing limit cycles}

Periodic solutions with constant sign are  determined by a linear differential equation, and so it is easy to prove that there are at most two (crossing) limit cycles with constant sign. 
 
\begin{proposition}\label{prop:0}
    Equation \eqref{eq:main2} has at most two limit cycles with constant sign.
\end{proposition}
\begin{proof}
Consider all solutions with positive sign (analogous for negative sign) for $t \in [0,2\pi]$. Let $u(t,x)$ be the solution such that $u(0,x) = x$. Then, 
\[
    u(2\pi,x) = x\,e^{ \int_0^{2\pi} a^+(s) \,ds} + \int_0^{\,2\pi} \, b^+(s)\, e^{ \int_{s}^{\,2\pi} a^+(\tau)\,d\tau}ds = Ax + B
\] for some $A,\,B \in \mathbb{R}$. If the solution is periodic, then $u(2\pi,x) = x = Ax + B$. As a consequence, no periodic solutions exist if $A=1$ and $B\neq 0$; all solutions are periodic if $A=1$ and $B=0$; and only one periodic solution exists if $A \neq 1$. Taking into account the unique negative periodic solution as well, we obtain the desired result. 
\end{proof}


By Proposition~\ref{prop:0}, the matter of interest are crossing limit cycles that change sign, so in the following we will refer to them simply as ``crossing limit cycles'', and we will reserve ``limit cycle'' for crossing limit cycles with constant sign.

The following lemma implies that, by a translation in time, we may assume 
that $u(0) < 0$ for every crossing limit cycle.

\begin{lemma}
There exists $\bar{t} \in [0,2\pi]$ such that every crossing limit cycle satisfies $u(\bar{t}) < 0$. 
\end{lemma}
\begin{proof}
As the vector field is Lipschitz continuous in $x\neq 0$, any crossing solution is uniformly bounded in $[0,2 \pi]$. In particular, the set of crossing limit cycles is bounded.

If the set of crossing limit cycles is not empty, then there exists a maximal element, $u_M$
. Moreover, it must be a periodic function, either positive or changing sign. 

If $u_M(t)\geq 0$ for every $t\in[0,2\pi]$, as it is the limit of crossing limit cycles, there exists $\bar t \in[0,2\pi]$ such that $u_M(\bar t)=0$. As it is maximal, any crossing limit cycle $u$ satisfies $u(\bar t)<u_M(\bar t)=0$. 

If $u_M$ changes sign, then there exists $\bar t \in[0,2\pi]$ such that $u_M(\bar t)<0$. As a consequence, for any crossing limit cycle $u$, $u(\bar t)\leq u_M(\bar t)<0$.
\end{proof}

Therefore, from now on, we assume that any crossing limit cycle, $u$, satisfies $u(0)<0$. The next step is to reduce the study of the number of crossing limit cycles to the number of simple zeros of a system of equations. 

\begin{proposition}\label{prop:1}
Let $u$ be a crossing limit cycle with $k$ zeros in $(0,2\pi)$ ($k$ even), denoted by $0<t_1<\ldots<t_k<2\pi$. Then $(t_1,\ldots ,t_{k})$, is a solution of the following system of equations:
    \begin{equation}\label{Ec3}
    \begin{split}
        \int_{t_1}^{t_2} b^+(t)\exp \left( \int_t^{t_2}a^+(s)ds \right) dt = 0, \\
        \int_{t_2}^{t_3} b^-(t)\exp \left( \int_t^{t_3}a^-(s)ds \right) dt = 0, \\
        \vdots \\
        \int_{t_{k-1}}^{t_{k}} b^+(t)\exp \left( \int_t^{t_{k}}a^+(s)ds \right) dt = 0, \\
        \int_{t_{k}}^{t_1 + 2\pi} b^-(t)\exp \left( \int_t^{t_1 + 2\pi}a^-(s)ds \right) dt = 0.
    \end{split}
\end{equation}
\end{proposition}

\begin{proof}
As $u$ is a crossing solution, then $b^+(t_1) b^-(t_1)>0$, but $t_1>0$ is the first zero and $u(0)<0$, so $b^+(t_1),b^-(t_1)>0$ and $u$ changes sign from negative to positive. In particular, $u(t)>0$ for $t\in(t_1,t_2)$ and it is a solution of $x'=a^+(t)x+b^+(t)$ such that $u(t_1)=u(t_2)=0$, so $t_1,t_2$ satisfy 
\[\int_{t_1}^{t_2} b^+(t)\exp \left( \int_t^{t_2}a^+(s)ds \right) dt = 0.
\]
Next, realize that $u$ is negative for $t\in (t_2,t_3)$, so
\[
\int_{t_2}^{t_3} b^-(t)\exp \left( \int_t^{t_3}a^-(s)ds \right) dt = 0.
\]
If we continue with the process for all of the zeros, we have that $t_i,t_{i+1}$ satisfy the corresponding equation for $1\leq i<k$. Finally, as $u$ is periodic, $u(t_1 + 2\pi)=0$. As a consequence, 
\[
\int_{t_{k}}^{t_1 + 2\pi} b^-(t)\exp \left( \int_t^{t_1 + 2\pi}a^-(s)ds \right) dt = 0.
\]

\end{proof}

If $u$ is a crossing limit cycle, then the number of zeros in $(0,2\pi)$ is bounded by $2m$, where $m = \max \{ \deg(b^+),\deg(b^-) \}$. The following proposition implies that the zeros of $u$ are also a solution of system \eqref{Ec3} with $k=2m$ equations.

\begin{lemma}\label{lemma:4}
If a crossing limit cycle $u$ of \eqref{eq:main2} has $l < 2m$ zeros, $t_1,\dots,t_l$, then the tuple $(t_1,\dots,t_{2m})$ such that $t_i = t_{l}$ for $i=l+1,\ldots,2m$ satisfies system~\eqref{Ec3} for $k=2m$.
\end{lemma}
\begin{proof}
The zeros $t_1,\dots,t_l$ constitute a solution for the first $l-1$ equations from system~\eqref{Ec3} with $k=2m$. As $t_{2m}=t_l$, the last equation also follows. Finally, as for $i=l+1,\dots,2m,\, t_i = t_l$, the remaining equations from system~\eqref{Ec3} are trivially satisfied.
\end{proof}


Let $u$ be a crossing limit cycle. The {\em displacement map}, $x\to d(x)$, is defined as $d(x) = u(2\pi,x) - x$, where $u(t,x)$ is the solution determined by $u(t,0)=x$. $d$ is analytical in a neighborhood of $u(0)$ (see e.g. \cite[p. 39--40]{CoddingtonLevinson1955} or \cite[Proposition~2]{Tineo}), and $d(u(0))=0$. The crossing limit cycle is {\em simple} if $d'(u(0))\neq 0$.

Next we prove that simple crossing limit cycles are related to non-degenerate solutions of~\eqref{Ec3}.  

\begin{proposition}\label{prop:simple}
Let $u$ be a crossing limit cycle with (even) $k$  zeros in $(0,2\pi)$,  $0<t_1<t_2<\ldots<t_k<2\pi$. If $u$ is simple then  $(t_1,\ldots,t_k)$ is a non-degenerate solution of \eqref{Ec3}.

Moreover, the tuple $(t_1,\dots,t_{2m})$ such that $t_i = t_{k}$ for $i=k+1,\ldots,2m$ is also a non-degenerate solution of system~\eqref{Ec3} for $k=2m$.
\end{proposition}

\begin{proof}

The displacement function $d$ in a neighborhood of $u(0)$ is defined as 
\[
d(x) = T(t_k(t_{k-1}(\ldots t_2(t_1(x))\ldots)))-x,
\]
where the functions $t_1,\ldots,t_k,T$ are related to the times of the solution of \eqref{eq:main2} passing by two values, and are defined implicitly as
\begin{equation}\label{eq:retorno}
    \begin{split}
    x\,e^{ \int_0^{t_1} a^-(s) \,ds} + \int_0^{t_1} \, b^-(s)\, e^{ \int_{s}^{t_1} a^-(\tau)\,d\tau}ds=0, \\
        \int_{t_1}^{t_2} b^+(t)\exp \left( \int_t^{t_2}a^+(s)ds \right) dt = 0, \\
        \vdots \\
        \int_{t_{k-1}}^{t_{k}} b^+(t)\exp \left( \int_t^{t_{k}}a^+(s)ds \right) dt = 0, \\
        \int_{t_{k}}^{2\pi} b^-(t)\exp \left( \int_t^{2\pi}a^-(s)ds \right) dt = T.
    \end{split}
\end{equation}
We may write $t_1(x),\dots,t_i(t_{i-1}),\dots,T(t_k)$ for $i=2,\dots,k$, as the partial derivative of the $i$-th equation with respect to $t_i$ is not null, and so the Implicit Function Theorem applies.

Deriving with respect to $x$ in the first equation, one gets
\[
\begin{split}
&e^{ \int_0^{t_1(x)} a^-(s) \,ds} + x\,a^-(t_1(x))e^{ \int_0^{t_1(x)} a^-(s) \,ds}\,t_1'(x) + b(t_1(x))t_1'(x) \\
&+ a^-(t_1(x))\int_0^{t_1(x)} \, b^-(s)\, e^{ \int_{s}^{t_1} a^-(\tau)\,d\tau}ds\, t_1'(x),
\end{split}
\]
and using the first equation of \eqref{eq:retorno}, we get
\[
e^{ \int_0^{t_1(x)} a^-(s) \,ds}  +  b^-(t_1) t_1'(x) = 0.
\]
That is,
\[
t_1'(x) = \frac{-e^{ \int_0^{t_1(x)} a^-(s) \,ds}}{ b^-(t_1)}.
\]

Now consider $t_{j+1}(t_{j})$, defined by the $j$-th equation in \eqref{eq:retorno}. Derivating with respect to $t_j$, one obtains
\[
t_{j+1}'(t_j)
=
\frac{
b^{\sigma(j)}(t_j)
\exp\left(\int_{t_j}^{t_{j+1}} a^{\sigma(j)}(s)\,ds\right)
}{
b^{\sigma(j)}(t_{j+1})
},
\]
where $\sigma(j) = +$ for $j$ odd and $\sigma(j)=-$ for $j$ even.

Finally,
\[
T'(t_k)=
-b^-(t_k)\exp\left(\int_{t_k}^{2\pi} a^-(s)\,ds\right).
\]

Replacing the previous expressions, and using that $a^-$ is $2\pi$-periodic,
\[
\begin{split}
d'(x)
&= T'(t_k) t_k'(t_{k-1}) \ldots t_2'(t_1) t_1'(x) - 1\\
&=
\frac{b^-(t_k)}{b^-(t_1)} e^{
\int_{0}^{t_1} a^-(s)\,ds+\int_{t_k}^{2\pi} a^-(s)\,ds}
\prod_{j=1}^{k-1}
\frac{
b^{\sigma(j)}(t_{j})
e^{\int_{t_j}^{t_{j+1}} a^{\sigma(j)}(s)\,ds}
}{
b^{\sigma(j)}(t_{j+1})
} -1\\
&=
\frac{b^-(t_k)e^{\int_{t_k}^{t_1+2\pi} a^-(s)\,ds}}{b^-(t_1)} 
\prod_{j=1}^{k-1}
\frac{
b^{\sigma(j)}(t_{j})
e^{\int_{t_j}^{t_{j+1}} a^{\sigma(j)}(s)\,ds}
}{
b^{\sigma(j)}(t_{j+1})
} -1.
\end{split}
\]

Assume that $d'(x) \neq 0$, and let us prove that the Jacobian of~\eqref{Ec3} in the solution $(t_1,\ldots,t_k)$ is an invertible matrix. 
A simple computation shows that 
the Jacobian matrix is the cyclic bidiagonal matrix 
\[
J=
\begin{pmatrix}
-d_1 & c_1 & 0 & \cdots & 0 \\
0 & -d_2 & c_2 & \cdots & 0\\
\vdots && \ddots & \ddots & \vdots \\
0 & \cdots & 0 & -d_{k-1} & c_{k-1}\\
c_k & 0 & \cdots & 0 & -d_k
\end{pmatrix},
\]
where for $j=1,\dots,k-1$,
\[
d_j
=
b^{\sigma(j)}(t_j)
\exp\!\left(\int_{t_j}^{t_{j+1}} a^{\sigma(j)}(s)\,ds\right),
\qquad
c_j = b^{\sigma(j)}(t_{j+1}),
\]
and
\[
d_k
=
b^{-}(t_k)
\exp\!\left(\int_{t_k}^{t_1+2\pi} a^{-}(s)\,ds\right),
\qquad
c_k = b^{-}(t_1+2\pi)=b^{-}(t_1).
\]

The determinant is (we recall $k$ is even)
\[
\det(J)
=
d_1 d_2 \cdots d_k
-
c_1 c_2 \cdots c_k
=c_1 c_2 \cdots c_k d'(x),
\]
so it is not null provided that $d'(x)\neq 0$ ($c_j=b^{\sigma(j)}(t_j)\neq 0$ as we are considering crossing limit cycles).

Finally, observe that if we consider the system in Lemma~\ref{lemma:4}, it only adds an even number of equations with derivatives of the form
\[
d_j = b^-(t_k) ,\quad c_j = b^-(t_k).
\]
Therefore, they do not affect the computation of $d'(x)$ and just multiply $\det(J)$ by a positive number.
\end{proof}

Next, we prove that it is sufficient to consider simple crossing limit cycles when obtaining an upper bound. To that end, consider the family
\begin{equation}\label{Ec3par}
x' = S(t,x)+\lambda,\quad \lambda\in\mathbb{R}.    
\end{equation}
Note that for every $\lambda\in\mathbb{R}$, \eqref{Ec3par} is of the form~\eqref{eq:main2}. As the vector field is monotonous with respect to $\lambda$, if $d(x,\lambda)$ is the displacement map in terms of $\lambda$, then $d_\lambda(x,\lambda)>0$, and the Implicit Function Theorem implies that $d(x,\lambda)=0$ defines a function $\lambda(x)$ locally.

\begin{proposition}\label{prop:simple_2zonas}
If \eqref{eq:main2} has $N$ crossing limit cycles, then  \eqref{Ec3par}  has at least $N$ simple crossing limit cycles for $\lambda$ sufficiently small.    
\end{proposition}
\begin{proof}

Let $u(t)$ be a crossing limit cycle for $\lambda = 0$. Then $u(0)$ is an isolated zero of the  map $x \to d(x,0)$. Therefore, $u(0)$ is a local maximum, a local minimum, or a change of sign of $\lambda(x)$. 

Assume $u(0)$ is a simple zero, that is, $d_x(u(0),0)\neq 0$. As $\lambda(x)$ is defined implicitly as $d(x,\lambda(x))=0$, deriving one obtains
\[0=d'(x,\lambda(x)) = d_x(x,\lambda(x)) + d_{\lambda}(x,\lambda(x))\lambda'(x),\]
\[\lambda'(u(0)) = -d_x(u(0),0)/d_{\lambda}(x,0) \neq 0.\]
So $\lambda(x)$ is locally a monotonous function. In consequence, locally, there is exactly one crossing limit cycle for $\lambda$ small enough, and it is simple for $\lambda$ sufficiently small. 

Now, assume that $\bar x=u(0)$ is not simple, so in particular $\lambda'(\bar x)=0$.

If $\bar x$ is a change of sign of $\lambda$, then for each sufficiently small fixed $\lambda$, the function $x \to d(x,\lambda)$ changes sign near $\bar x$, so it has locally exactly one limit cycle. Moreover, by analyticity, it is simple for $\lambda$ sufficiently small.

If $\bar x$ is an extremum of $\lambda$, then taking $\lambda>0$ sufficiently enough if it is a minimum or $\lambda<0$ small enough if it is a maximum, we obtain two crossing limit cycles. 
Now, by conveniently selecting $\lambda$ small enough and with the convenient sign (positive if $d$ has more minima than maxima and negative otherwise), we get the desired result.

\end{proof}

\subsection{Khovanskiĭ theory}

Let us now state and prove the main result for this section. Let \[m = \max \{ \deg(b^+), \deg(b^-) \},\quad N = \max \{m,\deg(a^+), \deg(a^-)\}.\]

\begin{theorem}\label{thm:B}
The number of limit cycles of~\eqref{eq:main2} for $m,N$ fixed, $\mathcal{H}(1,m,N)$,  satisfies  
\[
\mathcal{H}(1,m,N) \le 
2^{4m(8m-1)}\,(6Nm+2m+1)^{8m} + 2.
\]
\end{theorem} 
\begin{proof}
By previous discussion, we may assume that any crossing limit cycle, $u$, has $2m$ zeros, is simple, and satisfies $u(0) < 0$. From Proposition \ref{prop:1}, we get that if $t_1\leq \dots\leq t_{2m}$ are the zeros of $u$ in $(0,2\pi)$, then $(t_1,\dots,t_{2m})$ is a non-degenerate solution of \eqref{Ec3}, which can be written as follows after some computations:
\begin{equation}\label{mainsystem}
\begin{split}
\int_{t_1}^{t_2} f^+(\tau)\,d\tau=0,\\
\int_{t_2}^{t_3} f^-(\tau)\,d\tau=0,\\
\vdots \\
\int_{t_{k-1}}^{t_{2m}} f^+(\tau)\,d\tau=0,\\
\int_{t_{2m}}^{t_1+2\pi} f^-(\tau)\,d\tau=0,\\
\end{split}
\end{equation}
where 
\[
f^+(\tau) = b^+(\tau) \exp\left(-\int_0^{\tau} a^+(s)\,ds\right),\quad
f^-(\tau) = b^-(\tau) \exp\left(-\int_0^{\tau} a^-(s)\,ds\right).
\]

From Lemma \ref{lemma:0}, we know that all the equations in the system are Pfaffian. Nevertheless, we need some additional functions for the Pfaffian chain, since in equation \eqref{eq:main2} we have two different pieces with two different polynomials for each region. Let us define a Pfaffian chain for system \eqref{mainsystem} based on the proof of Lemma \ref{lemma:0}.

Let $ t = (t_1,\dots,t_{2m})$ and define the functions  \[f_{1_j}(t) = \tan\left(\frac{t_j}{2}\right), \qquad f_{2_j}(t) = \frac{1}{1+f_{1_j}^2(t)},\]
where $j = 1,\dots,2m$. These functions are part of the Pfaffian chain, since 
\[
\begin{split}
    d f_{1_j}(t) &= \dfrac{1 + f_{1_j}^2(t)}{2}\,dt_j, \\
     d f_{2_j}(t) &= -f_{1_j}(t)f_{2_j}(t)\,dt_j,
\end{split}
\]
that is, their derivatives can be expressed as quadratic polynomials in terms of the function and previous functions of the chain.
Next, we need to define different exponential functions
\[ f_{3_j}(t) = \exp \left( - \int_0^{t_j} a^{\sigma(j)}(s)\, ds \right), \]
where $a^{\sigma(j)}=a^+$ when $j$ is odd and $a^{\sigma(j)}=a^-$ when $j$ is even. As it was proved earlier, these functions are part of the chain, since
\[
d f_{3_j}(t) = -f_{3_j}(t)a^{\sigma(j)}(t_j)\,dt_j, 
\]
\[
a^{\sigma(j)}(t_j)=
a_0^{\sigma(j)} + \sum_{k=1}^N a_k^{\sigma(j)} Q_k(f_{1_j}(t))f_{2_j}^k(t) + \sum_{k=1}^N \tilde{a}_k^{\sigma(j)} P_k(f_{1_j}(t))f_{2_j}^k(t),
\]
where $a_i^{\sigma(j)},\tilde a^{\sigma(j)}_i$ are the coefficients of $a^{\sigma(j)}$ as a trigonometric polynomial (see Lemma~\ref{lemma:0} for more details).
Note that $df_{3_j}$ has degree $3N+1$ as a polynomial on the functions of the chain for $j=1,\dots,2 m$.
Last, we need $2m$ functions 
\[ 
f_{4_j}(t) = \int_{t_j}^{t_{j+1}} f^{\sigma(j)}(s)\,ds, 
\]
 where $j=1,\dots,2m$, being $t_{2m+1} = t_1 + 2\pi$, $f^{\sigma(j)}=f^+$ if $j$ is odd, and
 $f^{\sigma(j)}=f^-$ if $j$ is even.
These functions are part of the chain, since taking an odd $j$ (analogous for an even $j$), we get 
\[
\begin{split}
df_{4_j}(t) = &\dfrac{\partial f_{4_j}(t)}{\partial t_j}d t_j + \dfrac{\partial f_{4_j}(t)}{\partial t_{j+1}} d t_{j+1}
\\
=&-b^+(t_j)f_{3_j}(t)d t_j + b^+(t_{j+1})f_{3_{j+1}}(t) d t_{j+1}.
\end{split}
\]
Recall (see Lemma~\ref{lemma:0}) that $b^+$ 
can be written in terms of previous functions
as a polynomial of degree $3m\leq 3N$.

In conclusion, we need a Pfaffian chain of length $4 \cdot 2m = 8m$, with degree $\alpha = 3N+1$.

Let us now use the formula provided by Khovanskiĭ for a system of equations of Pfaffian functions (see Theorem~\ref{Khov}),
\begin{equation*}
        \mathcal{M}(n,r,\alpha,\beta_1,\ldots,\beta_n) := 2^{r(r-1)/2}\beta_1\dots \beta_n (\min \{n,r\}\alpha + \beta_1 + \dots + \beta_n - n + 1)^r,
\end{equation*}
to obtain an upper bound of the number of solutions of \eqref{mainsystem}.

Therefore, when \eqref{mainsystem} involves $n=2m$ variables, we need a chain of length $r= 8m$, the degree of which is equal to $\alpha = 3N+1$, and each equation has degree $\beta_i=1, \, $ for $i=1,\ldots,2m$ with respect to the chain. Therefore, the number of non-degenerate solutions of system \eqref{mainsystem} in the region $(-\pi,\pi)^{2m}$ is
\[
2^{\frac{8m(8m-1)}{2}}\,(2m\cdot(3N+1)+1)^{8m}.
\]

Now, since solutions of system \eqref{mainsystem} have $(0,2\pi)^{2m}$ as its domain, by a translation of the region $(-\pi,\pi)^{2m}$ we are able to cover the original domain. Last, we also add the limit cycles that have constant sign (see Proposition \ref{prop:1}), and we obtain the final bound:
\[
2^{4m(8m-1)}\,(6Nm+2m+1)^{8m} + 2.
\]
\end{proof}

\begin{remark}\label{remark1}
Khovanskiĭ's theorem provides an upper bound on the number of \emph{non-degenerate} real solutions of system \eqref{mainsystem}. However, this raw bound may substantially overcount the solutions that actually correspond to distinct limit cycles for the following reasons:
\begin{itemize}
    \item \textbf{Ordering of zeros.} All zeros of a limit cycle $t_1,\ldots,t_{2m}$ verify that $t_i \leq t_j$ with $i<j$. Hence, all solutions $(t_1,\ldots,t_{2m})$ of~\eqref{mainsystem} where some $t_i > t_j$ with $i<j$ do not correspond to a limit cycle.
    \item \textbf{Uniqueness constraints.} If we take two different solutions of~\eqref{mainsystem} but such that one element is equal, then only one of the solutions corresponds with a limit cycle, due to the uniqueness of solutions.
    \item \textbf{Repeated zeros and degeneracies.} Limit cycles with fewer than $2m$ zeros can be repeated and counted several times. By Lemma \ref{lemma:4}, these limit cycles have some of their zeros repeated, and by permuting the position of the repeated values, one can get more solutions related to the same limit cycle.
    \item \textbf{Realization as a limit cycle.} Even if we have a solution of~\eqref{mainsystem}, \(t_1,\ldots,t_{2m}\), we still require that the corresponding solution segment starting at \((t_1,0)\) not only ends at \((t_2,0)\), but also remains positive for every \(t\in(t_1,t_2)\). The same requirement applies to the remaining segments.
\end{itemize}
\end{remark}

All the equations in \eqref{mainsystem} arise from integrals of trigonometric polynomials and exponentials of trigonometric polynomials; in particular, the equations system belong to the same low-complexity functional family. In contrast, in order to give his upper bound in his theorem, Khovanskiĭ considers a broader and potentially heterogeneous class of Pfaffian functions that can be different within the system. Restricting to a homogeneous, low-dimensional generator set of functions (the same trigonometric and exponential functions appearing in every equation) typically allows strictly smaller upper bounds than those given by the most general Khovanskiĭ estimate.

This phenomenon is illustrated  in \cite{Ojeda}, where the authors obtained a smaller upper bound for \eqref{eq:main2} with 
\[a^{\pm}(t) =\pm\left(a_0 + a_1 \cos(t) + a_2\sin(t)\right),\quad b(t) = b_0 + b_1\cos(t) + b_2\sin(t),\]
where $ a_0,a_1,a_2,b_0,b_1,b_2 \in \mathbb{R}$, by exploiting the shared trigonometric structure rather than applying a fully general Pfaffian bound.

In conclusion, the direct application of Khovanski\u{\i}'s general bound yields a robust \emph{a priori} estimate, but it is not necessarily optimal for the specific trigonometric--exponential equations arising in~\eqref{mainsystem}.

\section{Linear piecewise ODE with an arbitrary number of regions}

In this section, we consider the general piecewise-linear equation on the cylinder with trigonometric polynomial coefficients,
\begin{equation}\label{eq:mainPCW2}
x' =
\begin{cases}
a_{n+1}(t)\,x +b_{n+1}(t), & x \ge x_{n},\\
a_{i}(t)\,x +b_{i}(t), & x_{i-1}\le x \le x_{\,i},\quad i=2,\dots,n,\\
a_{1}(t)\,x +b_{1}(t), & x \le x_{1},
\end{cases}
\end{equation}
where $a_i,\,b_i$, $i=1,\ldots,n+1$, are trigonometric polynomials. Denote $M$ the maximum degree of the trigonometric polynomials
 $a_i,\,b_i$, $i=1,\ldots,n+1$.

As we mentioned in the introduction, we will only consider crossing limit cycles, so if $u$ is a limit cycle and $u(\bar t) = x_k$ for some $\bar t\in[0,2\pi]$, then 
\[\left( a_k(\bar t)x_k + b_k(\bar t) \right) \left(  a_{k+1}(\bar t)x_k + b_{k+1}(\bar t) \right) > 0.\]

\subsection{Structure of the set of crossing limit cycles}

The structure of the set of crossing solutions of~\eqref{eq:mainPCW2} is studied in the following propositions. We will bound the number of crossings, and the number of possible crossing configurations with the splitting lines $x=x_i$, $1\leq i\leq n$. Finally, we will reduce the problem of bounding the number of limit cycles with a prescribed crossing configuration to counting non-degenerate solutions of a Pfaffian system of equations.


\begin{proposition}

    
A crossing limit cycle of~\eqref{eq:mainPCW2} has at most $2Mn$ crossings.
\end{proposition}
\begin{proof}



For each $i=1,\dots,n$, the function $t \rightarrow a_i(t)x_i + b_i(t)$ is a trigonometric polynomial of degree lower or equal than $M$. If one of the functions is identically zero, no crossing limit cycle crosses $x=x_i$, so we may assume none of these functions is identically zero. Then, each has at most $2M$ zeros in any period.  


Now, let $u$ be a crossing limit cycle of~\eqref{eq:mainPCW2}. The vector field must change sign between two consecutive crossings with the same line $x = x_i$. Hence, at each level $x = x_i$, $u$ can have at most $2M$ crossings, which yields a total of at most $2Mn$ crossings. 
\end{proof}

The following example shows that we can find piecewise-linear ODEs with solutions exhibiting the maximum number of crossings.

\begin{example} 
Consider~\eqref{eq:mainPCW2} with \[a_i=0,\quad  b_i(t) =-M(n+1)\sin(M\,t), \quad x_i = i,\quad i=1,\dots,n+1 .\] Then the solutions are $u(t)= (n+1)\cos(M\,t)+C$, $\,C\in\mathbb{R}$. Choosing $C=(n+1)/2$, it is easy to check that $u(t)$, $t\in[0,2\pi)$, has $2M$ simple crossings with each $x_i$, $i=1,\ldots,n$. Moreover, by transversality and continuity of the solutions, this number of crossings is preserved by small perturbations of $a_1,\ldots,a_{n+1},b_1,\dots,b_{n+1}$, so it is easy to obtain an equation with a crossing limit cycle with the maximum number of crossings.
\end{example}

Assume $u$ is a crossing limit cycle, and let $\mathcal{S}_1$ be the unit circle which we can identify to $[0,2\pi)$. We define the following maps: 
\[
T(u):=(t_1,\ldots,t_k)\in\mathcal{S}_1^k,
\]
where $t_1,\ldots,t_k\in \mathcal{S}_1$ are the ordered times when the solution crosses one of the splitting lines $\{x_1,\ldots,x_n\}$, that is, $u(t_i) = x_{j_i} \in \{x_1,\ldots,x_n\}$. 
The $k$-tuple $(t_1,\ldots,t_k)$ is understood up to cyclic permutation, that is, we consider the crossing times as an ordered $k$-tuple on $\mathcal S_1$ with no distinguished first element. The value $k \in \mathbb{N}$ is even, since the crossing limit cycles are periodic, so when they cross a splitting line $x \in \{x_1,\ldots,x_n\}$, they cross it in pairs.

Define also
\[
X(u):=(u(t_1),\ldots,u(t_k))\in \{x_1,\ldots,x_n\}^k,
\]
where $(t_1,\ldots,t_k)$ are the crossing times associated with $u$ as in the definition of $T(u)$, and the $k$-tuple $X(u)$ is understood up to cyclic permutation.

We prove that the set of initial conditions of crossing solutions with the same crossing sequence is an open set.

\begin{lemma}\label{lemma:3}
Let $u$ be a crossing limit cycle.
The set of solutions $v$ that are transversal to the splitting set $\{x_1,\dots,x_n\}$ on $[0,2\pi]$ and satisfy $X(v)=X(u)$
define an open subset of the space of initial conditions.
\end{lemma}
\begin{proof}
Let $v$ be a transversal solution such that $X(v)=X(u)$. Since $v$ is transversal, all its crossings with the splitting lines are transverse.
Therefore, for each crossing time $t_j$ there exists a neighborhood in time where the
solution crosses the same splitting line with the same orientation.

By continuous dependence of solutions on initial conditions, there exists a neighborhood
of the initial condition of $v$ such that any solution $w$ starting in this neighborhood
remains close to $v$ on $[0,2\pi]$. In particular, $w$ has exactly one crossing near each
$t_j$, crosses the same splitting line $x_{i_j}$, and no additional crossings appear.

Hence, the number, order, and location of the crossings are preserved, and therefore
$X(w)=X(v)=X(u)$ for all such solutions $w$. This proves the lemma.
\end{proof}

\color{black}

We prove now that, for a fixed~\eqref{eq:mainPCW2}, the number of possible sequences $X(u)$ is finite and bounded by $2Mn$.

\begin{proposition}\label{prop:2}
For equation~\eqref{eq:mainPCW2}, there are at most $2Mn$ possible sequences $X(u)$ up to cyclic permutation, where $u$ ranges over all crossing limit cycles.
\end{proposition}
\begin{proof}
Consider the connected component containing $u$ of the set of transversal solutions
$v$ of \eqref{eq:mainPCW2} restricted to $t\in [0,2\pi]$ and satisfying $X(v)=X(u)$.
As shown in Lemma \ref{lemma:3}, this set is open. Moreover, it can be naturally parametrized by the initial condition $v(0)$ and is therefore homeomorphic to an open interval of $\mathbb{R}$.

This connected component is bounded. Indeed, for initial conditions with sufficiently
large modulus, the corresponding solutions do not cross any of the splitting lines
$\{x_1,\ldots,x_n\}$ on $[0,2\pi]$, and hence cannot belong to the set under consideration.

Let $v_-$ and $v_+$ denote the infimum and supremum, respectively, of this connected
component in terms of the initial condition. By construction, neither $v_-$ nor $v_+$
can be transversal, since otherwise the component could be extended beyond them.
Therefore, both $v_-$ and $v_+$ must exhibit a tangency with one of the splitting lines.

Equivalently, the vector field $S(t,x)$ admits a zero at the corresponding level $x=x_i$
for some $i\in\{1,\ldots,n\}$ and at some time $t\in[0,2\pi]$. Thus, each connected component
of transversal solutions with fixed crossing sequence $X(u)$ is bounded by two zeros of
the vector field $S(t,x)$ occurring on two lines $x=x_{j_1}$, $x=x_{j_2}$.

If these zeros are simple, they cannot serve as boundary points of more than one connected
component. Indeed, at a simple zero on a splitting line $x=x_i$, one of the two lateral
vector fields changes sign at the zero, while the other does not. As a consequence, the
transversality condition can be satisfied only on one side of the zero and necessarily
fails on the other. Therefore, each zero can bound at most one connected component of
transversal solutions. Consequently, each pair of zeros gives rise to at most one connected
component of transversal solutions.

It follows that the total number of connected components of transversal solutions on
$[0,2\pi]$, and hence the total number of admissible crossing sequences $X(u)$, is bounded
by half the total number of zeros of $S(t,x)$ on the splitting lines $x=x_i$.
Since for each line $x=x_i$ there are at most $4M$ zeros, we conclude that the number of possible crossing sequences $X(u)$ is bounded by $2Mn$.
\end{proof}

Fix a $k$-tuple $(x_{i_1},\dots,x_{i_k})\in\{x_1,\ldots,x_n\}^k$ belonging to the image of the map $X$ restricted to crossing limit cycles; that is,
\[
(x_{i_1},\dots,x_{i_k}) = X(u)
\]
for some crossing limit cycle $u$.  
Observe that the sequence $X(u)$ also determines the intervals (pieces) in which the solution evolves between successive crossings. Indeed,
let $v$ be another crossing limit cycle such that $X(v)=X(u)$. By definition, there exist times
$t_1<\ldots<t_k$ satisfying
\[
v(t_j) = x_{i_j}, \qquad 1\le j\le k,
\]
and $v(t)\not\in\{x_1,\ldots,x_n\}$ 
for $t\in (t_1,t_k)$, $t\not\in \{t_1,\ldots,t_k\}$.
For each $j$, the behavior of $v(t)$ on the interval $(t_j,t_{j+1})$ is determined as follows:

\begin{enumerate}
\item If $x_{i_j}<x_{i_{j+1}}$, then
\[
x_{i_j} < v(t) < x_{i_{j+1}} \qquad \text{for all } t\in(t_{j},t_{j+1}).
\]

\item If $x_{i_j}>x_{i_{j+1}}$, then
\[
x_{i_{j+1}} < v(t) < x_{i_j} \qquad \text{for all } t\in(t_{j},t_{j+1}).
\]

\item If $x_{i_j}=x_{i_{j+1}}$, the interval containing the solution depends on the previous crossings. More precisely:
\begin{enumerate}
\item If $x_{i_{j-1}}<x_{i_j}$, then
\[
x_{i_j} < v(t) \qquad \text{for all } t\in(t_{j},t_{j+1}).
\]

\item If $x_{i_{j-1}}>x_{i_j}$, then
\[
x_{i_j} > v(t) \qquad \text{for all } t\in(t_j,t_{j+1}).
\]

\item If $x_{i_{j-1}}=x_{i_j}$, one proceeds backwards until encountering a strict inequality and applies the corresponding case above. In this situation, the solution alternates between intervals above and below $x_{i_j}$. If all crossings coincide, then their number must be even and the solution alternates above and below $x_{i_1}$.
\end{enumerate}
\end{enumerate}

In summary, once the sequence $X(u)$ is fixed, there exists a well-defined map assigning to each transition $x_{i_j}\to x_{i_{j+1}}$ an interval of $\mathbb{R}\setminus\{x_1,\ldots,x_n\}$ in which the corresponding segment of any crossing limit cycle must lie. Now, let us define a ``transition'' equation for a crossing solution of \eqref{eq:mainPCW2} that connects consecutive crossings at $(t_{i_j},x_{i_j})$ and $(t_{j+1},x_{i_{j+1}})$:
\begin{equation}\label{eq:transition2}
\begin{split}
\mathcal{Q}_j(s_1,s_2) =&
x_{i_j}\,e^{-\int_0^{s_1} a(s)\,ds}
+ \int_{s_1}^{s_2} b(t)\, e^{-\int_0^t a(s)\,ds}\,dt
\\&- x_{i_{j+1}}\,e^{-\int_0^{s_2} a(s)\,ds},
\end{split}
\end{equation}
where $a(t)$ and $b(t)$ denote the coefficients of \eqref{eq:mainPCW2} along the time interval $(t_j,t_{j+1})$, that is, the coefficients corresponding to the piece of the vector field visited by the solution between two consecutive crossing times. In particular, if $u$ is a crossing limit cycle such that $u(t_j) = x_{i_j}$, $u(t_{j+1}) = x_{i_{j+1}}$ for some $t_{j},t_{j+1}$, and $u(t)\not \in  \{x_1,\ldots,x_k\}$ for $t\in(t_{j},t_{j+1})$, then $\mathcal{Q}_j(t_{j},t_{j+1}) = 0$.
\medskip

For simplicity of notation, we write $a(t)$ and $b(t)$ for the coefficients of equation~\eqref{eq:mainPCW2} when the solution remains within a fixed region 
$(x_i,x_{i+1})$.
When the solution crosses to a different region, the corresponding coefficients $a(t)$ and $b(t)$ are understood to change accordingly.

\begin{proposition}\label{prop:5}
Let $u$ be a crossing limit cycle with
$T(u) = (t_1,\dots,t_k)$, $X(u) = (x_{i_1},\dots,x_{i_k})$. Define $Q_j$ by \eqref{eq:transition2}, and $t_{k+1} := t_1 + 2\pi$, $x_{i_{k+1}}:=x_{i_1}$.

Then, $T(u)$ is a solution of the following system of equations:
\begin{equation}\label{system2}
\begin{cases}
\mathcal{Q}_1(s_1,s_2) = 0, 
\\[0.3em]
\mathcal{Q}_2(s_2,s_3) = 0,
\\
\quad\vdots
\\
\mathcal{Q}_{k-1}(s_{k-1},s_{k}) = 0,
\\[0.3em]
\mathcal{Q}_k(s_k,s_{k+1}) = 0.
\end{cases}
\end{equation}
\end{proposition}

\begin{proof}
Firstly, we prove that $Q_1(t_1,t_2)=0$, being analogous for $Q_2(t_2,t_3)$, $\ldots$, $Q_{k-1}(t_{k-1},t_k)$. Since $u(t_1) = x_{i_1}$, by direct integration, one has  
\[
u(t) =  x_{i_{1}}e^{ \int_{t_1}^{t} a(s) ds} + \int_{t_1}^{t} b(s)e^{\int_{s}^t a(\tau)d\tau} ds,
\quad t\in(t_1,t_2).
\]
Now, using that $u(t_2) = x_{i_2}$, one obtains
\[
x_{i_2} =  x_{i_{1}}e^{ \int_{t_1}^{t_2} a(s) ds} + \int_{t_1}^{t_2} b(s)e^{\int_{s}^{t_2} a(\tau)d\tau} ds.
\]
Multiplying by $e^{ -\int_{0}^{t_2} a(s) ds}$,
we obtain $Q_1(t_1,t_2)=0$.

By periodicity, one has that $u(t_1+2\pi)=x_{i_1}$. Therefore, considering $t_{k+1}= t_1 + 2\pi$, $x_{i_{k+1}}=x_{i_1}$ and repeating arguments above, one has $Q_k(t_k,t_{k+1})=0$.

\end{proof}

Note that if $u,v$ are two different crossing limit cycles with $X(u)=X(v)$, then $T(u),T(v)$ are two different solutions of~\eqref{system2}. Moreover, arguing analogously to Propositions~\ref{prop:simple} and \ref{prop:simple_2zonas}, they may be assumed to be non-degenerate. 

\begin{proposition}\label{prop:6}
The number of crossing limit cycles of~\eqref{eq:main2} is bounded by the number of non-degenerate solutions of system~\eqref{system2}.    
\end{proposition}

\begin{proof}
We will follow the proofs of Propositions~\ref{prop:simple} and \ref{prop:simple_2zonas} and detail just the differences. 

Firstly, as in Proposition~\ref{prop:simple}, we will prove if $u$ is a simple cycle, then $T(u)$ is a non-degenerate solution of~\eqref{system2}.

The displacement function $d$ in a neighborhood of $u(0)$ is defined as 
\[
d(x) = T(t_k(t_{k-1}(\ldots t_2(t_1(x))\ldots)))-x,
\]
where the functions $t_1,\ldots,t_k,T$ are related to the times of the solution of \eqref{eq:mainPCW2} passing by two values, and are defined implicitly as
\begin{equation}\label{eq:retorno2}
\begin{split}
    x + \int_0^{t_1} \, b(t)\, e^{ -\int_{0}^{t} a(s)\,ds}dt &=x_{i_1} \,e^{-\int_0^{t_1} a(s) \,ds}, \\
x_{i_1}\,e^{-\int_0^{t_1} a(s)\,ds}
+ \int_{t_1}^{t_2} b(t)\, e^{-\int_0^t a(s)\,ds}\,dt
&= x_{i_2}\,e^{-\int_0^{t_2} a(s)\,ds}, \\
\vdots \\
x_{i_{k-1}}\,e^{-\int_0^{t_{k-1}} a(s)\,ds}
+ \int_{t_{k-1}}^{t_k} b(t)\, e^{-\int_0^t a(s)\,ds}\,dt
&= x_{i_k}\,e^{-\int_0^{t_k} a(s)\,ds}, \\
x_{i_k}\,e^{-\int_0^{t_k} a(s)\,ds}
+ \int_{t_k}^{2\pi} b(t)\, e^{-\int_0^t a(s)\,ds}\,dt
&= T\,e^{-\int_0^{2\pi} a(s)\,ds}.
\end{split}
\end{equation}

As we did in the proof of Proposition \ref{prop:simple}, deriving with respect to $x$ in the first equation of \eqref{eq:retorno2}, one gets
\[
t_1'(x) = \frac{-1}{a(t_1)x_{i_1}+b(t_1)}
\exp\left( \int_0^{t_1(x)} a(s) \,ds\right).
\]
Considering now $t_{j+1}(t_{j})$, defined by the $j$-th equation in \eqref{eq:retorno2}, deriving with respect to $t_j$ and computing, one obtains
\[
t_{j+1}'(t_j) = \frac{
x_{i_j}a(t_j)+b(t_j)
}{x_{i_{j+1}}a(t_{j+1}) + b(t_{j+1})}\exp\left(\int_{t_j}^{t_{j+1}} a(s)\,ds\right).
\]

Finally,
\[
T'(t_k)= -\left(x_{i_k}a(t_k)
+b(t_k)\right)\exp\left(\int_{t_k}^{2\pi} a(s)\,ds\right).
\]
Replacing the previous expressions and using that, by periodicity of $u$, the index of the function $a$ coincides at the origin and at $2\pi$, then, as it is $2\pi$-periodic,
\[
\int_0^{t_1(x)} a(s) \,ds + \int_{t_k}^{2\pi} a(s)\,ds=
\int_{t_k}^{t_1(x)+2\pi} a(s) \,ds ,
\]
so one obtains
\[
\begin{split}
d'(x)
&= T'(t_k) t_k'(t_{k-1}) \ldots t_2'(t_1) t_1'(x) - 1 \\
&=
\prod_{j=1}^{k}
\frac{
x_{i_j}a(t_j)+b(t_{j})
}{
x_{i_{j+1}}a(t_{j+1})+b(t_{j+1})
}\exp\left(\int_{t_j}^{t_{j+1}} a(s)\,ds\right)
 -1,
\end{split}
\]
where $x_{i_{k+1}}=x_1$, and $t_{k+1} = t_1+2\pi$ (we have used that $a,b$ are periodic to write $a(t_{k+1})$, $b(t_{k+1})$ instead $a(t_{1})$, $b(t_{1})$).

Assume that $d'(x) \neq 0$, and let us prove that the Jacobian of~\eqref{system2} in the solution $(t_1,\ldots,t_k)$ is an invertible matrix. 
A simple computation shows that 
the Jacobian matrix is the cyclic bidiagonal matrix 
\[
J=
\begin{pmatrix}
-d_1 & c_1 & 0 & \cdots & 0 \\
0 & -d_2 & c_2 & \cdots & 0\\
\vdots && \ddots & \ddots & \vdots \\
0 & \cdots & 0 & -d_{k-1} & c_{k-1}\\
c_k & 0 & \cdots & 0 & -d_k
\end{pmatrix},
\]
where for $j=1,\dots,k-1$,
\[
\begin{split}
d_j
&= (a(t_j)x_{i_j} + b(t_j))\,e^{-\int_0^{t_j}a(s)\,ds},
\\
c_j &= (a(t_{j+1})x_{i_{j+1}} + b(t_{j+1}))\,e^{-\int_0^{t_{j+1}}a(s)\,ds},
\end{split}
\]
and, 
\[
\begin{split}
d_k &=  (a(t_k)x_{i_k} + b(t_k)) e^{-\int_0^{t_k} a(s)\,ds}, 
\\
c_k &= (a(t_{k+1})x_{i_{k+1}} + b(t_{k+1})) e^{-\int_0^{t_{k+1}}a(s)\,ds}.
\end{split}
\]

The determinant is (we recall $k$ is even)
\[
\det(J)
=
d_1 d_2 \cdots d_k
-
c_1 c_2 \cdots c_k
=c_1 c_2 \cdots c_k d'(x),
\]
so it is not null provided that $d'(x)\neq 0$ ($c_j\neq 0$ for all $j=1,\dots,k$, as we are considering crossing limit cycles).

Analogously with Proposition \ref{prop:simple_2zonas}, we may consider simple crossing limit cycles when obtaining an upper bound. In order to do so, we consider the family
\[
x' = S(t,x) + \lambda, \quad \lambda \in \mathbb{R}.
\]
As the vector field is monotonous with respect to $\lambda$, if $d(x,\lambda)$ is the displacement map in terms of $\lambda$, then $d_\lambda(x,\lambda)>0$. 

Let $u(t)$ be a crossing limit cycle for $\lambda = 0$. Then $u(0)$ is an isolated zero of the map $x \to d(x,0)$. Therefore, $u(0)$ is a local maximum, a local minimum, or a change of sign of $\lambda(x)$. By conveniently choosing $\lambda$ small enough and with the convenient sign, we get the desired result.
\end{proof}

All expressions that appear in equation \eqref{eq:transition2} and in system~\eqref{system2} can be written as combinations of Pfaffian functions involved in the proof of Lemma \ref{lemma:0}. Therefore, they belong to the Pfaffian chain, and the number of non-degenerate solutions of the system is finite and can be bounded using Khovanskiĭ’s theorem.

\subsection{Proof of the main result} Let us now restate and prove Theorem~\ref{thm:A}. Let 
\[
M = \max_{1\leq i \leq n+1} \{\deg(a_i), \deg(b_i)\}_{i=1}^{n+1},\]
and
\[
C(n,M) :=   2^{4Mn(8Mn-1)}\,(6M^2n + 2Mn+1 )^{8Mn}(2Mn) + 2.
\]

\begin{theoremA*}
Let $\mathcal{H}(n,M)$ be the supremum of the number of limit cycles of \eqref{eq:mainPCW2} with $n+1$ linear pieces in $x$ and trigonometric polynomial coefficients in $t$ of degree at most $M$. Then
\[
\mathcal{H}(n,M) \le C(n,M).
\]
\end{theoremA*}
\begin{proof}[Proof of Theorem~\ref{thm:A}]
By Proposition~\ref{prop:2}, there are at most $2Mn$ possible sequences $X(u)$ up to cyclic permutation, where $u$ ranges over all crossing limit cycles. Let us fix one of them. For any crossing limit cycle $u$ with the prescribed sequence $X(u)$,  Proposition~\ref{prop:5} states that $T(u)$ is a solution of~\eqref{system2}. Moreover, by Proposition~\ref{prop:6} we may assume $T(u)$ is a non-degenerate. Then the number of crossing limit cycles with the prescribed $X(u)$ is bounded by the number of non-degenerated solutions of~\eqref{system2}.


From Proposition \ref{lemma:0}, we know that all the equations in the system are Pfaffian. Let us define a Pfaffian chain for system \eqref{system2} based on the proof of Lemma \ref{lemma:0}.

Let $ t = (t_1,\dots,t_{k})$ and define the functions  \[f_{1_j}(t) = \tan\left(\frac{t_j}{2}\right), \qquad f_{2_j}(t) = \frac{1}{1+f_{1_j}^2(t)},\]
where $j = 1,\dots,k$. These functions are part of a Pfaffian chain, since 
\[
\begin{split}
    d f_{1_j}(t) &= \dfrac{1 + f_{1_j}^2(t)}{2}\,dt_j, \\
     d f_{2_j}(t) &= -f_{1_j}(t)f_{2_j}(t)\,dt_j,
\end{split}
\]
that is, their derivatives can be expressed as quadratic polynomials in terms of the function and previous functions of the chain.
Next, we need to define different exponential functions
\[ f_{3_j}(t) = \exp \left( - \int_0^{t_j} a(s)\, ds \right), \]
where $a(s)$ denotes the corresponding coefficient $a_i(t)$ of \eqref{eq:mainPCW2} when the solution remains within a fixed region $[0,2\pi]\times(x_i,x_{i+1})$. As it was proved earlier, these functions are part of the chain, since
\[
d f_{3_j}(t) = -f_{3_j}(t)a(t_j)\,dt_j, 
\]
\[
a(t_j)=
a_0 + \sum_{l=1}^N a_k Q_l(f_{1_j}(t))f_{2_j}^l(t) + \sum_{l=1}^N \tilde{a}_l P_l(f_{1_j}(t))f_{2_j}^l(t),
\]
where $a_i,\tilde a_i$ are the coefficients of $a$ as a trigonometric polynomial (see Lemma~\ref{lemma:0} for more details).
Note that $df_{3_j}$ has degree $3N+1$ as a polynomial on the functions of the chain for $j=1,\dots,k$.
Last, we need $2m$ functions 
\[ 
f_{4_j}(t) = \int_{t_j}^{t_{j+1}}b(t)e^{- \int_0^ta(s)ds}\, dt,
\]
where $j=1,\dots,k$, being $t_{k+1} = t_1 + 2\pi$ and $b(t)$ denotes the corresponding coefficient $b_i(t)$ of \eqref{eq:mainPCW2} when the solution remains within a fixed region $[0,2\pi]\times(x_i,x_{i+1})$..
These functions are part of the chain, we get 
\[
df_{4_j}(t) = \dfrac{\partial f_{4_j}(t)}{\partial t_j}d t_j + \dfrac{\partial f_{4_j}(t)}{\partial t_{j+1}} d t_{j+1}
= -b(t_j)f_{3_j}(t_j)d t_j   -b(t_j)f_{3_j}(t_j) d t_{j+1}.
\]
Recall (see Lemma~\ref{lemma:0}) that $b$ 
can be written in terms of previous functions
as a polynomial of degree $3m\leq 3N$.
In conclusion, we need a Pfaffian chain of length $ r =4 \cdot k \leq  4 \cdot 2Mn = 8Mn$.
\medskip

If we substitute in the formula given by Khovanskiĭ (see Theorem \ref{Khov}),
\begin{equation*}
        \mathcal{M}(k,r,\alpha,\beta_1,\ldots,\beta_k) := 2^{r(r-1)/2}\beta_1\dots \beta_n (\min \{k,r\}\alpha + \beta_1 + \dots + \beta_k - k + 1)^r,
\end{equation*}
for $\beta_i = 1$, for $i=1,\ldots,k$, $\alpha = 3\max\{\deg(a_i),\, \deg(b_i) \}_{i=1}^{n+1} + 1 = 3M + 1$ and $k = 2Mn$, we get that the number of non-degenerated solutions of system \eqref{system2} in the set $(-\pi, \pi)^k$ is
\[
2^{\frac{8Mn(8Mn-1)}{2}}\,(2Mn(3M + 1)+1 )^{8Mn}.
\]
This value is the upper bound of the number of limit cycles with the prescribed $X(u)$. Since there are $2Mn$ possible $X(u)$, we need to multiply by this number.

In addition, since solutions of equation \eqref{eq:mainPCW2} have $(0,2\pi)^k$ as their domain, by a translation of the region $(-\pi,\pi)^k$ we can cover the original domain. Last, we also add the limit cycles with constant sign that lie outside the set $[0,2\pi]\times (x_1,x_n)$, and so we obtain the final bound:
    \[ \mathcal{H}(n,M) \le
2Mn \left( 2^{4Mn(8Mn-1)}\,(6M^2n + 2Mn+1 )^{8Mn}\right) + 2.
\] 
\end{proof}

\begin{remark}
As described in remark \ref{remark1}, Khovanskiĭ's Theorem provides an upper bound on the number of non-degenerate real solutions of system \eqref{system2}, which do not necessarily relate to a limit cycle. Therefore, the actual upper bound for the number of limit cycles of \eqref{eq:mainPCW2} is much smaller.


Furthermore, in Section 3 we consider the same system of equations \eqref{Ec3} for all crossing limit cycles with the maximum number of crossings, even though many solutions have fewer crossings. As a result,  we obtained one upper bound that overestimates the number of such cycles, but considered only one configuration of limit cycles. 

This strategy cannot be considered in Section 4, since the system of equations \eqref{system2} changes whenever the sequence $X(u)$ changes, due to the presence of multiple regions. Therefore, the bound for a sequence $X(u)$ must be multiplied by $2Mn$ (the maximum number of crossings of a limit cycle) in order to bound the number of crossing limit cycles for all possible sequences $X(u)$.

As a consequence, the bound obtained in Section 4 for the same conditions as in Section 3 (piecewise-linear ODE with two regions) is bigger:
\[
\begin{split}
\mathcal{H}(1,M,M) &\le 
2^{4M(8M-1)}\,(6M^2+2M+1)^{8M} + 2 
    \\ &< C(1,M) = 2^{4M(8M-1)}\,(6M^2 + 2M+1 )^{8M}(2M) + 2.
\end{split}
\]
\end{remark}

\subsection{Continuous vector field}

We complete this section proving that the upper bound also applies when the vector field $S$ is continuous. Note that limit cycles with a tangency at $x=0$ may exist when $S$ is continuous, which were not considered in the upper bound. We prove that by perturbing $S$, we obtain an equation with at least the same number of limit cycles, and such that all of which have transversal crossings. In consequence, the upper bound obtained in Theorem \ref{thm:A} is also an upper bound on the number of limit cycles of~\eqref{eq:main2} with $S$ continuous.

\begin{proposition}
If $S$ is continuous, then the number of limit cycles of~\eqref{eq:main2} is lower or equal than
$\mathcal{H}(n,M)$.
\end{proposition}
\begin{proof}
The idea of the proof is to show that we can consider only limit cycles with no tangencies. We will mimic the proof of Proposition~\ref{prop:6}. To that end,
fix a continuous $S(t,x)$ and consider the family of differential equations
\begin{equation}\label{eq:lambda}
x' = S(t,x) + \lambda .    
\end{equation}
The vector field of~\eqref{eq:lambda} is monotonous with respect to $\lambda$ (some authors call this ``a rotated family with respect to $\lambda$'' by analogy with the planar case).

Assume $u$ is a solution of~\eqref{eq:mainPCW2} tangent to a level $x=x_i\in \{x_1,\dots,x_n \}$. That is, there exists $\bar t\in[0,2\pi)$ such that $u(\bar t)= x_i$ and $u'(\bar t)=0$. In particular, $S(\bar t,x_i) = 0$. Therefore, each zero of $S(t,x)$ corresponds to a solution with at least one tangency. 

Let $\bar t\in[0,2\pi)$ be a zero of $S(t,x)$, and consider the family of solutions $u(t,\lambda)$ of~\eqref{eq:lambda} satisfying the initial condition $x(\bar t,\lambda)=0$. As the vector field is monotonous, there exists at most one value of $\lambda\in\mathbb{R}$ such that $u(t,\lambda)$ is periodic. In particular, there is a finite number of values of $\lambda$ with a limit cycle with a tangency. We shall prove that we can choose a small perturbation of $\lambda$ such that the number of limit cycles is at least increased, and therefore, as all are now crossing limit cycles, the upper bound holds also for $\lambda=0$.

Assume $\bar x$ is an isolated zero of the displacement map $x\to d(x,0)$ such that $d$ changes sign. Then $d$ changes sign for nearby values of $\lambda$ for any sufficiently small neighbourhood of $\bar x$. So for any sufficiently small value of $\lambda$, there is at least one limit cycle with initial condition close to $\bar x$.

Now, suppose that $\bar x$ is an isolated zero of $x\to d(x,0)$ such that $d$ does not change sign for $\lambda=0$; for instance, $d(x, 0 ) \geq 0$ for $x$ in a neighbourhood of $\bar x$, the inequality being strict for $\bar x \neq x$. Since the vector field is monotonous with respect to $\lambda$, then for $\lambda < 0$ sufficiently small, $d$ has at least two zeros close to $\bar x$. 

Now, by conveniently selecting $\lambda$ small enough and with the convenient sign (positive if $d$ has more minima than maxima and negative otherwise), we get the desired result.
\end{proof}

\color{black}

\end{document}